\documentclass{amsart}

\usepackage[english]{babel}
\usepackage[utf8]{inputenc}
\usepackage{amsmath}
\usepackage{amsthm}
\usepackage{amsfonts}
\usepackage{graphicx}
\usepackage[colorinlistoftodos]{todonotes}

\newtheorem{theorem}{Theorem}[section]
\newtheorem{corollary}{Corollary}[theorem]
\newtheorem{lemma}[theorem]{Lemma}
\newtheorem{proposition}[theorem]{Proposition}
\newtheorem{definition}[theorem]{Definition}
\newcommand{\gon}{\mathrm{gon}}
\newcommand{\val}{\mathrm{val}}
\newcommand{\tw}{\mathrm{tw}}
\newcommand{\Cliff}{\mathrm{Cliff}}
\newcommand{\Jac}{\mathrm{Jac}}

\title{Gonality of Random Graphs}

\author{Andrew Deveau}
\author{David Jensen}
\author{Jenna Kainic}
\author{Dan Mitropolsky}

\date{}

\begin{document}
\maketitle
\begin{abstract}
We show that the expected gonality of a random graph is asymptotic to the number of vertices.
\end{abstract}

\section{Introduction}

In the moduli space of curves, the locus of Brill-Noether general curves is a dense open subset \cite{GriffithsHarris80}.  In the moduli space of tropical curves, the Brill-Noether general locus is open \cite{LPP12,Len12} and non-empty \cite{tropicalBN}, but it is not dense \cite{Heawood}.  A natural question, therefore, is \emph{how likely is it that a graph is Brill-Noether general?}

In this paper, we approach this question by studying the gonality of Erd\"{o}s-R\'{e}nyi random graphs.  This is a natural follow-up to other recent work on the divisor theory of random graphs.  Most notably, in \cite{Lor08}, Lorenzini asks about the distribution of divisor class groups of random graphs, and in \cite{CLP13} it is conjectured that they are distributed according to a variation of the Cohen-Lenstra heuristics.  This conjecture is proved in \cite{Wood14}, expanding on the preliminary work of \cite{CKLPW14}.

Recall that a divisor $D$ on a graph $G$ has positive rank if $D-v$ is equivalent to an effective divisor for all vertices $v$ in $G$.  The gonality of a graph is the smallest degree of a divisor with positive rank.  For a more detailed account of divisor theory and gonality of graphs, see \cite{BakerNorine07} and \cite{Baker08}.  Our main result is the following.

\begin{theorem}
\label{thm:mainthm}
Let $p(n) = \frac{c(n)}{n}$, and suppose that $c(n) \ll n$ is unbounded.  Then
$$ \mathbb{E} ( \gon (G(n,p))) \sim n . $$
\end{theorem}

Theorem \ref{thm:mainthm} essentially says that the expected gonality of a random graph is as high as possible.  We note, however, that random graphs are not Brill-Noether general, as the genus of a random graph is asymptotically $\frac{c(n)n}{2}$, and if $c(n)$ is unbounded, this grows faster than $n$.  From this perspective, it may be more natural to study the gonality of random \emph{regular} graphs, as the genus of such graphs grows in proportion to the number of vertices.  The case of 3-regular graphs would be particularly interesting, as such graphs correspond to top-dimensional strata of the moduli space of tropical curves.

At the time of writing, we became aware of simultaneous work by Amini and Kool, in which they use an improvement on the spectral methods of \cite{CFK13} to show that the gonality of a random graph is bounded above and below by constant multiples of $n$ \cite{AK14}.  Our results are essentially a tightening of these bounds, so that both upper and lower bounds are asymptotic to $n$, which indeed is conjectured in \cite[Section 5.2]{AK14}.  The techniques of \cite{AK14} apply additionally to metric graphs, which we do not discuss here, and to the case of random regular graphs, which they show to have gonality bounded above and below by constant multiples of $n$ as well.

Also of note is the bound that we provide on the error term $n - \mathbb{E} ( \gon (G(n,p)))$ (see Theorem \ref{thm:UpperBound}).  In the future, it would be interesting to explore with what precision we can bound this term.

A more complete study of the Brill-Noether theory of random graphs would involve divisors of rank greater than one.  A natural generalization of the current line of inquiry would be to study the Clifford index of random graphs, defined as
$$ \Cliff (G) := \min_{D \in \Jac (G)} \{ \deg (D) - 2r(D) \vert r(D) > 0 \text{ and } r(K_G - D) > 0 \} . $$
Note that, if the minimum in this expression is obtained by a divisor of rank one, then $\Cliff (G) = \gon (G) -2$.  The Clifford index of an algebraic curve $C$ is known to always be either $\gon (C) - 2$ or $\gon (C) - 3$ \cite{CM91}.  The corresponding statement remains open for graphs, but if true, it would imply that the Clifford index of a random graph is asymptotic to the number of vertices as well.

\subsection*{Acknowledgements}

This paper was written as part of the 2014 Summr Undergraduate Math Research at Yale (SUMRY) program.  We would like to extend our thanks to everyone involved in the program, and in particular to Sam Payne, who suggested this project.  We also thank Matt Kahle for a particularly fruitful discussion.

\section{A Lower Bound}

In this section, we obtain a lower bound on the expected gonality of a random graph.  The first step is to identify a lower bound for the gonality of an aribitrary graph.  This is done in \cite{dBG}, where it is shown that the \emph{treewidth} of a graph is a lower bound for the gonality.

\begin{definition}
A \emph{tree decomposition} of a graph $G$ is a tree $T$ whose nodes are subsets of the vertices of $G$, satisfying the following properties:
\begin{enumerate}
\item  Each vertex of $G$ is contained in at least one node of $T$.
\item  If two nodes of $T$ both contain a given vertex $v$, then all nodes of the tree in the unique path between these two nodes must contain $v$ as well.
\item  If two vertices $v$ and $w$ are adjacent in $G$, then there is a node of $T$ that contains both $v$ and $w$.
\end{enumerate}
The \emph{width} of a tree decomposition is one less than the number of vertices in its largest node.  The \emph{treewidth} $\tw (G)$ of a graph $G$ is the minimum width among all possible tree decompositions of G.
\end{definition}

\begin{proposition} \cite{dBG}
\label{prop:treewidth}
Let $G$ be a simple graph.  Then
$$ \gon(G) \geq \tw(G). $$
\end{proposition}

Although we will not use it, we note the following simple consequence.

\begin{corollary}
For a simple graph $G$,
$$ \gon(G) \geq \min \{ \val (v) \vert v \in V(G) \} . $$
\end{corollary}

\begin{proof}
The result follows immediately from Proposition \ref{prop:treewidth} and the fact that $\tw (G) \geq \min \{ \val (v) \vert v \in V(G) \}$ (see \cite{BK11}).
\end{proof}

The treewidth of random graphs has been studied extensively in \cite{WLCX11} and \cite{Gao12}.

\begin{lemma}\cite{WLCX11}
\label{lem:ExpectedTreewidth}
Let $p(n) = \frac{c(n)}{n}$, and suppose that $c(n) \ll n$ is unbounded.  Then
$$ \tw (G(n,p)) \geq n - o(n) $$
with high probability.
\end{lemma}

\begin{theorem}
\label{thm:LowerBound}
Let $p(n) = \frac{c(n)}{n}$, and suppose that $c(n) \ll n$ is unbounded.  Then
$$ \lim_{n \to \infty}\mathbb{P}(\gon (G(n,p)) \geq n - o(n)) = 1 . $$
\end{theorem}

\begin{proof}
By Lemma \ref{lem:ExpectedTreewidth}, the treewidth is bounded below with high probability by $n - o(n)$.  Furthermore, by Proposition \ref{prop:treewidth}, we know that gonality is bounded below by the treewidth, so with high probability the gonality is also bounded below by $n - o(n)$.
\end{proof}



\section{An Upper Bound}

In this section, we obtain an upper bound on the expected gonality of a random graph.  Together with the results of the previous section, this will imply that the gonality of a random graph is asymptotically equal to the number of vertices.  We note that the number of vertices $n$ is a very simple upper bound for the gonality of a graph, and together with Theorem \ref{thm:LowerBound}, this would be enough to establish the main theorem.  We actually go a bit further and obtain a bound on the expected value of $n - \gon (G(n,p))$.  In the future, it would be interesting to explore this with higher precision.

Recall that an \emph{independent set} in a graph is a set of vertices, no pair of which are connected by an edge.  The \emph{independence number} $\alpha (G)$ of a graph $G$ is defined to be the maximal size of an independent set.

\begin{theorem}
\label{thm:independence}
If $G$ is a simple graph with $n$ vertices, then $\gon (G) \leq n - \alpha (G)$.
\end{theorem}

\begin{proof}
Let $I$ be a maximal independent set, and let $D$ be the sum of the vertices in the complement of $I$.  We will show that $D$ has positive rank.  If $v \notin I$, then $D-v$ is effective by definition.  On the other hand, if $v \in I$, then since all of the neighbors of $v$ are not in $I$ and the graph is simple, by firing all of the vertices other than $v$ we obtain an effective divisor equivalent to $D$ with at least one chip on $v$.  It follows that $D$ has rank at least one, hence $\gon (G) \leq \deg (D) = n - \alpha (G)$.
\end{proof}

Note that gonality $n - 1$ is achieved by the complete graph $K_n$, so this bound is sharp.  Note further that the complete graph is the only simple graph with $n$ vertices whose gonality is $n-1$.

The expected independence number of a random graph has been studied in \cite{Frieze90}.

\begin{lemma}\cite{Frieze90}
\label{lem:Frieze}
Let $p(n) = \frac{c(n)}{n}$, and suppose that $c(n) \ll n$ is unbounded..  For any $\epsilon > 0$, we have
$$ \lim_{n \to \infty}\mathbb{P}(|\alpha(G(n,p))-\frac{2}{p(n)}(\log c(n) - \log\log c(n) - \log 2 + 1)| \leq \frac{\epsilon}{p(n)}) = 1 . $$
\end{lemma}

From this, we can conclude the following.

\begin{theorem}
\label{thm:UpperBound}
Let $p(n) = \frac{c(n)}{n}$, and suppose that $c(n) \ll n$ is unbounded.  Then
$$ \lim_{n \to \infty}\mathbb{P}(\gon (G(n,p)) \leq n - \frac{2}{p(n)}(\log c(n) - \log\log c(n) - \log 2 + 1)) = 1 . $$
\end{theorem}
\begin{proof}
By Lemma \ref{lem:Frieze}, for any $\epsilon > 0$, we have
$$ \alpha (G(n,p)) > \frac{2}{p(n)}(\log c(n) - \log\log c(n) - \log 2 + 1 - \epsilon ) $$
with probability 1 as $n$ approaches infinity.  By Theorem \ref{thm:independence}, the number $n - \alpha (G(n,p))$ is an upper bound for the gonality of $G(n,p)$.
\end{proof}

\begin{proof}[Proof of Theorem~\ref{thm:mainthm}]
By Theorem \ref{thm:UpperBound}, the gonality of a random graph is bounded above by $n-o(n)$.  Similarly, by Theorem \ref{thm:LowerBound}, the gonality of a random graph is bounded below by $n-o(n)$.  It follows that
$$ \lim_{n \to \infty} \frac{1}{n} \mathbb{E} (\gon (G(n,p))) = \lim_{n \to \infty} \frac{n-o(n)}{n} = 1 . $$
\end{proof}

\bibliographystyle{alpha}
\bibliography{mybib}
\end{document}